\documentclass[reqno,a4paper,12pt]{amsart} 

\usepackage{amsmath,amscd,amsfonts,amssymb}
\usepackage{mathrsfs,dsfont}

\numberwithin{equation}{section}
\numberwithin{figure}{section}

\newcommand\R{\mathbb{R}}
\newcommand\C{\mathbb{C}}

\newcommand\Z{\mathbb{Z}}

\newcommand\al{\alpha}
\newcommand\gam{\gamma}
\newcommand\Gam{\Gamma}
\newcommand\lam{\lambda}
\newcommand\Lam{\Lambda}
\newcommand\del{\delta}

\newcommand\om{\omega}

\newcommand\eps{\varepsilon}

\renewcommand\le{\leqslant}
\renewcommand\ge{\geqslant}
\renewcommand\leq{\leqslant}
\renewcommand\geq{\geqslant}
\newcommand\sbt{\subset}

\renewcommand\hat{\widehat}

\renewcommand\Im{\operatorname{Im}}

\newcommand{\ft}[1]{\widehat #1}

\newcommand{\supp}{\operatorname{supp}}

\renewcommand\S{\mathcal{S}}
\newcommand\DK{\underline{{D}}}

\theoremstyle{plain}
\newtheorem{thm}{Theorem}[section]
\newtheorem{lem}[thm]{Lemma}

\newtheorem*{claim*}{Claim}

\newcommand{\thmref}[1]{Theorem~\ref{#1}}

\newcommand{\lemref}[1]{Lemma~\ref{#1}}

\newcommand{\exampref}[1]{Example~\ref{#1}}

\theoremstyle{definition}

\newtheorem*{definition*}{Definition}
\newtheorem*{remarks*}{Remarks}
\newtheorem*{remark*}{Remark}

\newtheorem{example}[thm]{Example}

\newenvironment{enumerate-roman}
{\begin{enumerate}
\addtolength{\itemsep}{5pt}
}
{\end{enumerate}}

\newenvironment{enumerate-alph}
{\begin{enumerate}
\addtolength{\itemsep}{5pt}
}
{\end{enumerate}}

\newenvironment{enumerate-num}
{\begin{enumerate}
\addtolength{\itemsep}{5pt}
}
{\end{enumerate}}

\newenvironment{enumerate-text}
{\begin{enumerate}
\addtolength{\itemsep}{5pt}
}
{\end{enumerate}}

\begin{document}

\title[Crystalline temperate distributions]
{Crystalline temperate distributions with uniformly discrete support and spectrum}

\author{Nir Lev}
\address{Department of Mathematics, Bar-Ilan University, Ramat-Gan 5290002, Israel}
\email{levnir@math.biu.ac.il}

\author{Gilad Reti}
\address{Department of Mathematics, Bar-Ilan University, Ramat-Gan 5290002, Israel}
\email{gilad.reti@gmail.com}

\date{April 27, 2021}
\thanks{Research supported by ISF Grant No.\ 227/17 and ERC Starting Grant No.\ 713927.}

\begin{abstract}
We prove that a temperate distribution on $\R$ whose support and spectrum are uniformly discrete sets, can be obtained from Poisson's summation formula by a finite number of basic operations (shifts, modulations, differentiations, multiplication by polynomials, and taking linear combinations).
\end{abstract}

\maketitle


\section{Introduction} \label{secI1}

By a \emph{crystalline measure} on $\R$ (or $\R^d$) we
mean a 
pure point measure $\mu$ which is a
 temperate distribution and whose distributional
Fourier transform $\ft{\mu}$ is also a pure point measure,
\begin{equation}
  \label{eqI1.10}
\mu = \sum_{\lam\in \Lam} a(\lam) \del_\lam,
\quad
\ft{\mu} = \sum_{s \in S} b(s) \del_s,
\end{equation}
where the support $\Lam$ and the spectrum $S$
are locally finite sets \cite{Mey16}.
This notion may be considered as a 
mathematical model for quasicrystals, 
i.e.\ atomic arrangements having
a discrete diffraction pattern (see \cite{Mey95}, \cite{Lag00}).

A classical  example of a crystalline measure is
\begin{equation}
  \label{eqI1.20}
\mu = \sum_{\lam\in L} \del_\lam,
\end{equation}
where $L$ is a lattice. Indeed, by
 the Poisson summation formula, the Fourier
transform $\ft{\mu}$ is the sum of equal
masses on the dual lattice $L^*$.
By applying a finite number of
 shifts, modulations, and taking linear combinations, one
can construct more general examples 
of  crystalline measures $\mu$ whose
supports 
are contained in   finite unions of
 translates of the lattice $L$,
while the spectra contained
in   finite unions of
 translates of $L^*$.

However there exist also examples of crystalline measures
$\mu$ on $\R$, such that the 
support $\Lam$ is not contained in any finite
union of translates of a lattice. Constructions
of such examples, using different approaches, were
given in \cite{LO16}, \cite{Kol16}, 
\cite{Mey16}, \cite{Mey17}, \cite{RV19},
\cite{KS20},  \cite{Mey20}, \cite{OU20}.

On the other hand, it was proved in
\cite{LO13}, \cite{LO15} that if the
support $\Lam$ and the spectrum $S$ of a 
crystalline measure $\mu$ on $\R$ are 
\emph{uniformly discrete sets},
 then the measure can be obtained from  Poisson's 
 formula by a finite number of shifts, modulations, 
and taking linear combinations.

One can introduce a more general notion of
a \emph{crystalline temperate distribution},
which is by definition a temperate distribution $\alpha$ 
whose support
$\Lam$ is a locally finite set and such that
the spectrum $S$ (the support of the 
Fourier transform  $\ft{\alpha}$) is also 
a locally finite set.
See \cite{Pal17}, \cite{Fav18}, \cite{LR20}
where distributions of this type were studied.

Examples of crystalline temperate distributions
which are not crystalline measures
 may be constructed by starting from
 a crystalline measure $\mu$
and applying, in addition to the 
 operations mentioned above, also a finite number of 
\emph{differentiations} and
 \emph{multi\-plication by polynomials}.
The main result of  the present paper is 
that any temperate distribution 
on $\R$ whose support and spectrum are 
 uniformly discrete sets,
 can be obtained from Poisson's summation formula by 
these basic operations:

\begin{thm}
  \label{thmA1}
  Let $\alpha$ be a temperate distribution on $\R$
such that $\Lam = \supp(\alpha)$ and
$S = \supp(\ft{\alpha})$
are uniformly discrete sets. Then $\alpha$ can be
represented in  the form
  \begin{equation}
    \label{eqA1.1}
\alpha = \sum_{(\tau,\om,l,p)} c(\tau,\om,l,p)
\sum_{\lam \in L} \lam^l \, e^{2 \pi i \lam \om} \, \delta_{\lam+\tau}^{(p)}
  \end{equation}
where $L$ is a  lattice, 
$(\tau,\om,l,p)$ goes through a finite
set of quadruples such that
$\tau, \om$ are real numbers
and $l,p$ are nonnegative integers,
and $c(\tau,\om,l,p)$ are complex numbers.
\end{thm}

Conversely, if $\alpha$  is a 
 distribution of the form
\eqref{eqA1.1} then its support 
$\Lam$ and spectrum $S$ are uniformly
discrete sets. In fact, $\Lam$
is contained in  a finite union of
 translates of the lattice $L$, while 
$S$ is contained in a finite union of  translates 
of  the dual lattice  $L^*$.

The proof of 
\thmref{thmA1} is based on the approach 
developed  in \cite{LO13}, \cite{LO15}
which is generalized from the context
of measures to  temperate distributions.


\section{Preliminaries}

In this section we briefly recall some preliminary background
in the theory of Schwartz distributions
(see \cite{Rud91} for more details).

The \emph{Schwartz space}  $\S(\R)$ consists of all infinitely smooth
functions $\varphi$ on $\R$ such that for each $n,k \geq 0$, the norm
\begin{equation}
\|\varphi\|_{n,k} := \sup_{x \in \R} (1+|x|)^n \sum_{j =0}^{k} |\varphi^{(j)}(x)|
\end{equation}
is finite. 
A \emph{temperate distribution} on $\R$ is a 
linear functional  on the Schwartz space
which is continuous with respect to the 
 topology generated by 
this family of norms.

We use  $\alpha(\varphi)$ to denote  the action of
a  temperate distribution $\alpha$ on a
Schwartz function $\varphi$.
For each temperate distribution $\alpha$ there exist $n$ and $k$ such that
\begin{equation}
\label{eq:schnorm}
|\alpha(\varphi)| \leq C\|\varphi\|_{n,k},  \quad \varphi \in \S(\R),
\end{equation}
where $C = C(\alpha,n,k)$ is a constant  which  does not depend on $\varphi$.

If $\varphi$ is a
Schwartz function then its Fourier transform
is defined by
\begin{equation}
\hat{\varphi}(t) = \int_{\R} \varphi(x) \, e^{-2\pi i t x} \, dx.
\end{equation}
If $\alpha$ is a temperate distribution then
its Fourier transform is defined by 
$\ft{\alpha}(\varphi) = \alpha(\ft{\varphi})$.

We denote by $\supp(\alpha)$ the closed support of a
temperate distribution $\alpha$.

If $\alpha$ is
a temperate distribution  and if
$\varphi$  is a 
Schwartz function, then 
the product $\alpha \cdot \varphi$ is 
a temperate distribution   defined
by $(\alpha \cdot \varphi)(\psi) =
\alpha(\varphi \cdot \psi)$,
$\psi \in \S(\R)$.
If $\varphi$ does not vanish at any
point of $\supp(\alpha)$ then we have
$\supp(\alpha \cdot \varphi)=
\supp(\alpha)$.

The convolution
$\alpha \ast \varphi$ 
of a temperate distribution  
 $\alpha$ and 
a Schwartz function $\varphi$ is 
 an infinitely smooth function
which is also
a temperate distribution, and
whose Fourier transform is $\ft{\al} \cdot \ft{\varphi}$.
If $\varphi$ has compact support then
$\supp(\al \ast \varphi)$ is contained in the
 Minkowski sum 
 $\supp(\al) + \supp(\varphi)$.

\begin{lem}
  \label{lem:alphcoeffpolygrow}
Let  $\alpha$ be
a temperate distribution whose
support $\Lam$ is
a uniformly discrete set. Then 
$\alpha$ has the form
\begin{equation}
\label{eq:tudphiprod}
\alpha = \sum_{p=0}^{k} \sum_{\lam \in \Lam}
a_p(\lam) \del_{\lam}^{(p)}
\end{equation}
where the coefficients $a_p(\lam)$ satisfy the condition
\begin{equation}
\label{eq:tudsumcoefffin}
 \sum_{p=0}^{k} 
|a_p(\lam)|  \leq C(1 + |\lam|)^n,
\quad \lam \in \Lam,
\end{equation}
for certain constants $n$ and $C$.
\end{lem}

For a proof see
e.g.\ \cite[Proposition 2]{Pal17}
or \cite[Proposition 3.1]{Fav18}.


\section{Spectral gap and density}
\label{sec:gaps}

We say that a temperate distribution $\gam$ 
has a \emph{spectral gap} of size $a>0$, if 
its Fourier transform $\ft{\gam}$ vanishes on an 
 interval of length $a$.
There is a well-known principle stating that if
a set $\Gam \subset \R$ supports a nonzero 
measure, or a distribution,
with a spectral gap, then $\Gam$ cannot be ``too sparse''.
Several concrete versions
of this general principle can be found in
 \cite[Proposition 7]{KM58}, 
 \cite{MP10}, \cite{Pol12}, \cite[Section 4]{LO15}.

Let $\Gam \subset\mathbb R$ be a locally finite set
(that is, a set with no finite accumulation points) 
and consider a distribution $\gam$ of the form
\begin{equation}
\label{eq:tlamabsconv}
\gam = \sum_{p=0}^{k} \sum_{\lam \in \Gam}
c_p(\lam) \del_{\lam}^{(p)},
\end{equation}
where the coefficients $c_p(\lam)$ are
assumed to satisfy the condition
\begin{equation}
\label{eq:tlocfinsumcoefffin}
 \sum_{p=0}^{k} \sum_{\lam \in \Gam}
|c_p(\lam)| < + \infty.
\end{equation}
The condition \eqref{eq:tlocfinsumcoefffin} implies
in particular that the sum in \eqref{eq:tlamabsconv}  converges
in the space of temperate distributions.

We define the density $\DK(\Gam)$ of 
the  set $\Gam$ to be 
\begin{equation}
\label{eq:defdk}
\DK(\Gam) := \liminf_{R\to+\infty}
\frac1{2R}\int_{1}^{R} \frac{n_\Gam(r)}{r} dr,
\end{equation} 
where we denote $n_\Gam(r) := \# (\Gam \cap [-r,r])$.
We have the following result:

\begin{thm}
\label{thm:gapdens}
Let $\Gam \subset \R$ be a locally finite set,
and let $\gam$ be
a nonzero distribution of the form
\eqref{eq:tlamabsconv} and such that
condition \eqref{eq:tlocfinsumcoefffin} is satisfied.
If the Fourier transform $\ft{\gam}$
 vanishes on an interval of length $a$, then we
must have
\begin{equation}
\label{eq:gapdens}
\DK(\Gam)\ge a/(k+1).
\end{equation}
\end{thm}

The result actually holds under more general
assumptions, where instead of  conditions
\eqref{eq:tlamabsconv} and 
\eqref{eq:tlocfinsumcoefffin} one
merely assumes that $\gam$ is a
nonzero distribution 
with $\supp(\gam) \sbt \Gam$
and satisfying the condition
$|\gam(\varphi)| \leq C\|\varphi\|_{n,k}$, 
$\varphi \in \S(\R)$, where
 $n,k$ and $C$ do not depend on $\varphi$.
However we will not use this more
general version in the paper and we
do not give its proof.
The proof of \thmref{thm:gapdens} given below 
is based on 
a classical approach which involves an application
of Jensen's formula to the Cauchy transform of 
 the distribution $\gam$, see \cite[p.\ 73]{KM58}.

We note that 
the estimate  \eqref{eq:gapdens}  is sharp,
see \exampref{rem:gapdenssharp} below.

\begin{proof}[Proof of \thmref{thm:gapdens}]
Consider the Fourier transform $\ft{\gam}$ of the distribution $\gam$, 
\begin{equation}
\label{eq:fttexprsum}
\ft{\gam}(t)  = \sum_{p=0}^{k} (2 \pi i t)^p \sum_{\lam \in \Gam}
c_p(\lam) e^{-2 \pi i \lam t}.
\end{equation}
Then $\ft{\gam}$ is  a continuous function satisfying the 
estimate 
\begin{equation}
\label{eq:tftestpoly}
|\ft{\gam}(t)| \leq K (1+|t|)^k,
\end{equation}
where $K$ is a certain constant which does not depend on $t$.
Define the function
\begin{equation}
\label{eq:qtftypos}
f(z) := - \int_{0}^{\infty} \ft{\gam}(t)  \, e^{2 \pi i z t} \,dt,
\quad \Im(z)>0,
\end{equation}
and
\begin{equation}
\label{eq:qtftyneg}
f(z) := \int_{-\infty}^{0} \ft{\gam}(t)  \, e^{2 \pi i z t} \,dt,
\quad \Im(z)<0.
\end{equation}
The integrals in \eqref{eq:qtftypos} and \eqref{eq:qtftyneg}
converge absolutely, due to the 
 estimate \eqref{eq:tftestpoly}.

The function $f$ is called the
\emph{Fourier-Carleman transform} of the distribution $\gam$.
It follows from the 
uniqueness property of the Fourier transform 
that the function 
$\ft{\gam}$, and hence also the distribution $\gam$, 
are  uniquely determined by $f$.
In particular, since   $\gam$ is assumed
to be nonzero, it follows that
$f$ does not vanish identically.

If we substitute \eqref{eq:fttexprsum} into 
\eqref{eq:qtftypos} and \eqref{eq:qtftyneg}
and exchange the order of summation and integration
(which is justified  using \eqref{eq:tlocfinsumcoefffin}
and the dominated convergence theorem) then we obtain
\begin{equation}
\label{eq:fctrnsdef}
f(z) =  \frac1{2\pi i} 
\sum_{p=0}^{k} \sum_{\lam \in \Gam} c_p(\lam)
 \frac{p!  (-1)^p}{(z-\lam)^{p+1}}
\end{equation}
for every $z \in \C \setminus \R$.
Using again the assumption
\eqref{eq:tlocfinsumcoefffin} this
implies that $f$ can be extended to 
a meromorphic function in $\C$ whose
poles  are contained in $\Gam$, and
such that the multiplicity of each 
pole is at most $k+1$.

The right hand side
of \eqref{eq:fctrnsdef} is called
the \emph{Cauchy transform} of  the distribution $\gam$.
 The equality in \eqref{eq:fctrnsdef}
states a well-known relation between the Cauchy transform
and the Fourier-Carleman transform, see e.g.\ \cite[Section 3]{Ben84}.

We now use the assumption that
$\ft{\gam}$
 vanishes on an interval of length $a$.
By applying a translation we may assume, with
no loss of generality, that $\ft{\gam}$ vanishes
on the interval $(-a/2, a/2)$.
Using this together with
 \eqref{eq:tftestpoly}, \eqref{eq:qtftypos} 
and \eqref{eq:qtftyneg} we obtain
\begin{equation}
\label{eq:ctrnsest}
|f(x+iy)| \leq K \int_{a/2}^{\infty} (1+t^{k})  e^{-2 \pi |y| t} dt
=  K |y|^{-k-1} p_k(|y|) e^{- \pi a |y|},
\end{equation}
where $p_k$ is a polynomial of degree $k$ whose
coefficients depend on $k$ and $a$ but do not depend on $x$ and $y$.
This implies the estimate
\begin{equation}
\label{eq:qtdecayest}
|f(x+iy)| \leq C (|y|^{-k-1} + |y|^{-1}) e^{- \pi a |y|},
\quad x \in \R, \quad y \neq 0,
\end{equation}
where $C$ is a constant not depending on $x$ and $y$.

Finally we apply
Jensen's formula (see e.g.\ \cite[Chapter XII, Section 1]{Lan99}
or \cite[Sections 2.4 and 2.5]{Lev96}) to the function $f$. It
 yields that for $R \geq 1$ we have
\begin{equation}
\label{eq:jensenequal}
 \int_{1}^{R} \frac{n_f(r,0) - n_f(r,\infty)}{r} dr
= \int_{0}^{2 \pi} \log |f(R e^{i \theta})| \frac{d \theta}{2 \pi}
+  c_f,
\end{equation}
where $n_f(r,0)$ is the number of zeros 
(counted with multiplicities) of $f$ in the
disk $\{z : |z| \leq r\}$, 
$n_f(r,\infty)$ is the number of poles in the
same disk (again counted with multiplicities),
 and $c_f$ is a constant
which depends on $f$ but does
not depend on $R$.

The estimate 
\eqref{eq:qtdecayest}   implies that
\begin{equation}
\label{eq:festjens}
 \int_{0}^{2 \pi} \log |f(R e^{i \theta})|  \frac{d \theta}{2 \pi}
\leq  - 2 a R +  o(R), \quad R \to +\infty.
\end{equation}
We also have
\begin{equation}
\label{eq:polesleqkpone}
n_f(r, \infty) \leq (k+1) n_\Gam(r),
\end{equation}
since the poles of $f$ are contained in $\Gam$ and
the multiplicity
of each pole is at most $k+1$. It then follows from 
\eqref{eq:jensenequal},
\eqref{eq:festjens} and
\eqref{eq:polesleqkpone}
that
\begin{equation}
\label{eq:jenscombine}
 (k+1) \int_{1}^{R} \frac{n_\Gam(r)}{r} dr
\geq 2aR - o(R), \quad R \to +\infty.
\end{equation}
If we now divide both sides of  \eqref{eq:jenscombine}
by $2(k+1)R$ and take the $\liminf$ as $R \to +\infty$,
we arrive at \eqref{eq:gapdens}.
This concludes the proof of \thmref{thm:gapdens}.
\end{proof}

\begin{example}
  \label{rem:gapdenssharp}
The following example shows that the estimate 
 \eqref{eq:gapdens}  is sharp.
Let $\Gam := \Z$,
then $\DK(\Gam)=1$.
Given $\eps>0$ we 
construct a nonzero distribution $\gam$ 
satisfying \eqref{eq:tlamabsconv}
and \eqref{eq:tlocfinsumcoefffin}, and such
that the Fourier transform $\ft{\gam}$
vanishes on an
interval of length $a = k+1-\eps$. 
Let
 $\alpha := \sum_{p=0}^{k} a_p \sum_{\lam \in \Z}  \del_\lam^{(p)}$,
where we choose the coefficients $\{a_p\}$ 
 so that $P(t) := \sum_{p=0}^{k} a_p (2 \pi i t)^p$
is a nonzero polynomial 
vanishing on  the set $\{1,2,\dots,k\}$. Then  $\alpha$ is 
a nonzero temperate distribution, $\supp(\alpha) \sbt \Gam$,
and $\ft{\alpha}$
vanishes on the open interval $(0, k+1)$.
We then let $\gam := \alpha \cdot \varphi$,
where
$\varphi$ is a Schwartz function such that 
$|\varphi|>0$ and
$\supp(\ft\varphi) \sbt (-\eps,0)$.
Then the distribution $\gam$ 
is nonzero,  it satisfies the conditions
 \eqref{eq:tlamabsconv} and
\eqref{eq:tlocfinsumcoefffin},  and 
$\ft{\gam}$ vanishes on
the interval $[0, k+1-\eps]$.
\end{example}


\section{Arithmetic structure of the support}

In this section our goal is
to prove the following result:

\begin{thm}
  \label{thmB1}
  Let $\alpha$ be a temperate distribution on $\R$
such that $\Lam = \supp(\alpha)$ and
$S = \supp(\ft{\alpha})$
are both uniformly discrete sets. Then 
$\Lam$ is contained in a finite union
of translates of some lattice.
\end{thm}

The result was proved in 
\cite{LO13}, 
\cite{LO15}
in the case
where $\alpha$ 
and $\ft{\al}$
are measures, and it
is generalized here to the context
of   temperate distributions.
 \thmref{thmB1}
 constitutes the   first step in the
proof of \thmref{thmA1},
where the next step 
 consists of showing
that the  distribution $\al$ must be of
the form \eqref{eqA1.1}.

\subsection{}
Recall that
a set $\Lambda\subset \R$ is said to be
\emph{uniformly discrete} if there is
$\delta>0$ such that $|\lambda'-\lambda|\ge \delta$ for any two distinct points
$\lambda,\lambda'$ in $\Lambda$. The maximal constant $\delta$ with this property is
called the \emph{separation constant} of $\Lambda$, and will be denoted by
$\delta(\Lambda)$.

One can check that  if  $\Lambda$
is  a uniformly discrete set then 
$\DK(\Lam) \le 1/\delta(\Lambda)$.
In particular, the
density $\DK(\Lam)$ is finite.

We say that a set
 $\Lam \sbt \R$ is 
 \emph{relatively dense} if there 
is $a>0$ such that any interval $[x,x+a]$
 contains at least one point from $\Lam$.

If $\Lam$ is uniformly discrete and also
relatively dense, then $\Lam$ is called a
\emph{Delone set}.

\subsection{}
Let  $\Lam \sbt \R$, and for each $h\in \Lam-\Lam$ denote
\begin{equation}
\label{eq:lamhdef}
\Lam_h:=\Lambda\cap(\Lambda-h)=\{\lam \in \Lam : \lam+h \in \Lam\}.
\end{equation}
Then  $\Lam_h$ is a nonempty subset of $\Lam$.
We will use the following key result:

\begin{thm}[{\cite{LO13}, \cite{LO15}}]
  \label{thmLO1}
Let $\Lambda\subset \R$ be a Delone set,
and suppose that there exists a constant $c = c(\Lam) >0$
such that $\DK(\Lam_h) \ge c$ for
every $h \in \Lam-\Lam$. Then 
$\Lam$ is contained in a finite union
of translates of a certain lattice.
\end{thm}

This was actually proved
in \cite{LO13}, \cite{LO15}
with the density
\begin{equation}
\label{eq:dsharp}
D_{\#}(\Lam_h) := \liminf_{r \to +\infty}
\frac{\# (\Lam_h \cap [-r,r])}{2r}
\end{equation}
instead of $\DK(\Lam_h)$ in the
statement, but both the result and its proof
are valid for either one of these densities.
The proof, see \cite[Sections 5, 6]{LO15}, involves the concept of 
\emph{Meyer sets} and is based on
a characterization
of these sets that is due to Meyer  \cite{Mey72}.

\subsection{}
The following result can be found
in \cite[Proposition 7]{KM58}.

\begin{lem}
  \label{lem:suppreldens}
Let  $\alpha$ be
a nonzero temperate distribution whose
spectrum $S = \supp(\ft{\alpha})$ is 
 uniformly discrete.  Then 
the support  $\Lam = \supp(\alpha)$ is a relatively dense set.
\end{lem}

\begin{proof}
Suppose that $[x, x+a]$ is an interval of length $a$ disjoint
from the support $\Lam$. Let $\varphi$ be a 
Schwartz function supported on a sufficiently
small neighborhood of the
origin so that $\alpha \ast \varphi$ vanishes on $[x,x+a]$.
It follows from \lemref{lem:alphcoeffpolygrow} 
(applied to the Fourier transform $\ft{\al}$
of $\al$) that the distribution 
$\gam := \ft{\alpha} \cdot \ft{\varphi}$
satisfies the conditions
\eqref{eq:tlamabsconv} and 
\eqref{eq:tlocfinsumcoefffin} with $\Gam = S$
and with  $k$ that does not
depend on $\varphi$. 
If we choose $\varphi$ such that
also $\ft{\varphi} > 0$, then 
$\gam$ is a nonzero distribution
whose Fourier transform vanishes
on an interval of length $a$.
\thmref{thm:gapdens}  then yields
that $a$ cannot be
greater than $(k+1)\DK(S)$.
\end{proof}

\subsection{}
We  now 
use the previous results
in order to
prove \thmref{thmB1}.

\begin{proof}[Proof of \thmref{thmB1}]
By \lemref{lem:suppreldens} the support
$\Lam$ is  relatively dense, so
$\Lam$ is a Delone set. Due to
\thmref{thmLO1} it will therefore
suffice to show that there is $c >0$
such that $\DK(\Lam_h) \ge c$ for
every $h \in \Lam-\Lam$. 
We know from 
 \lemref{lem:alphcoeffpolygrow}
that there exist $n$ and $k$ such that the distribution
$\alpha$ has the form
\eqref{eq:tudphiprod}
and that
\eqref{eq:tudsumcoefffin}
is satisfied. We will prove
that the condition
\begin{equation}
\label{eq:dklamh}
\DK(\Lam_h) \geq \del(S)/(k+1)
\end{equation}
holds for  every $h \in \Lam - \Lam$.

To prove
\eqref{eq:dklamh}   we will
show that given any
$h \in \Lam - \Lam$ and any
$r>0$ one can find  
 a nonzero distribution 
\begin{equation}
\label{eq:ahform}
\gam_h  = \sum_{p=0}^{k} \sum_{\lam \in \Lam_h}
c_{p,h}(\lam) \del_{\lam}^{(p)}
\end{equation}
such that the coefficients $c_{p,h}(\lam)$ 
satisfy 
\begin{equation}
\label{eq:ahcoeff}
 \sum_{p=0}^{k} \sum_{\lam \in \Lam_h}
|c_{p,h}(\lam)| < + \infty,
\end{equation}
and  such that the Fourier transform
$\ft{\gam}_h$ vanishes on the set
\begin{equation}
\label{eq:usetdef}
U := \R \setminus [(S-S) + (-r,r)].
\end{equation}
Observe that 
the last property implies that the distribution
$\gam_h$ has a spectral
gap of size $\del(S) - 2r$, and so it follows from
 \thmref{thm:gapdens}  that
$\DK(\Lam_h) \geq (\del(S)-2r)/(k+1)$.
Since $r$ may be chosen arbitrarily small,
this yields 
\eqref{eq:dklamh}.

In order to construct the distribution $\gam_h$ we  
use the approach in \cite[Section 2.1]{LO13}.
We fix 
 $r>0$ and choose a Schwartz function $\varphi > 0$
such that  $\supp(\ft{\varphi}) \sbt (-r/2, r/2)$.
It follows from \eqref{eq:tudsumcoefffin}
that the product $\alpha \cdot \varphi$ 
has the form
\begin{equation}
\label{eq:Trexpr}
\alpha \cdot \varphi = \sum_{p=0}^{k} \sum_{\lam \in \Lam}
b_p(\lam) \del_{\lam}^{(p)}
\end{equation}
where the coefficients $b_p(\lam)$ satisfy
\begin{equation}
\label{eq:sumcoefffin}
 \sum_{p=0}^{k} \sum_{\lam \in \Lam}
|b_p(\lam)| < + \infty.
\end{equation}
(The coefficients $b_p(\lam)$ depend on the function
$\varphi$, but $k$ does not.)

Let $f$ be the Fourier transform of the distribution $\al \cdot \varphi$.
Then we have $f = \ft{\al} \ast \ft{\varphi}$ and hence
$f$ is an infinitely smooth function 
vanishing outside  the $(r/2)$-neighborhood of the
set $S$. On the other hand, by \eqref{eq:Trexpr} we have
\begin{equation}
\label{eq:fexpr}
f(x) = \sum_{p=0}^{k} (2 \pi i x)^p \sum_{\lam \in \Lam}
b_p(\lam) e^{-2 \pi i \lam x}.
\end{equation}
Let $g(x,u) := f(x) \overline{f(x-u)}$, then
using \eqref{eq:fexpr} and  opening the
brackets we obtain
\begin{equation}
\label{eq:gexpr}
g(x,u) = \sum_{l=0}^{2k} (2 \pi i x)^l
\sum_{h \in \Lam-\Lam} A_{h,l}(u) e^{-2 \pi i h x},
\end{equation}
where the coefficients $A_{h,l}(u)$  are given by
\begin{equation}
\label{eq:ahlexpr}
A_{h,l}(u) = \sum_{(p,q,j) \in J(k,l)} (-1)^{p+j} \binom{p}{j}
(2 \pi i u)^j \sum_{\lam \in \Lam_h} 
b_q(\lam+h) \overline{b_p(\lam)} e^{-2 \pi i \lam u}
\end{equation}
and where $J(k,l)$ denotes the finite set
of all triples $(p,q,j) \in \Z^3$  satisfying
\begin{equation}
\label{eq:nmjcond}
 0 \leq p,q \leq k, \quad 0 \leq j \leq p, \quad p+q-j = l.
\end{equation}

We observe that $g(x,u)$ is,
as a function of $x$, 
 the Fourier transform of the distribution
\begin{equation}
\label{eq:betaexpr}
\beta_u = \sum_{l=0}^{2k} \sum_{h \in \Lam-\Lam}
A_{h,l}(u)  \del_{h}^{(l)},
\end{equation}
and the coefficients $A_{h,l}(u)$ satisfy
\begin{equation}
\label{eq:betasumcoefffin}
 \sum_{l=0}^{2k} \sum_{h \in \Lam-\Lam}
|A_{h,l}(u)| < + \infty.
\end{equation}
On the other hand, 
for every $u \in U$ the function 
$g(x,u)$ vanishes identically with respect to $x$.
It  follows from the uniqueness of the
representation
\eqref{eq:betaexpr}
that $A_{h,l}(u)=0$, $u \in U$.

For each $h \in \Lam-\Lam$ and $0 \le l \le 2k$ we now define 
\begin{equation}
\label{eq:gamhldef}
\gam_{h,l} := 
\sum_{(p,q,j) \in J(k,l)} (-1)^{p+j} \binom{p}{j}
\sum_{\lam \in \Lam_h}   
b_q(\lam+h) \overline{b_p(\lam)}  \, \del_\lam^{(j)}.
\end{equation}
The fact that
\begin{equation}
\label{eq:gamhlcoefff}
\sum_{\lam \in \Lam_h}
| b_q(\lam+h) \overline{b_p(\lam)} |  < + \infty,
\quad (p,q,j) \in J(k,l),
\end{equation}
ensures that the sum \eqref{eq:gamhldef}
converges in the space of temperate distributions.
Moreover, the distribution $\gam_{h,l}$ has the form
\eqref{eq:ahform}, \eqref{eq:ahcoeff},
since the condition $(p,q,j) \in J(k,l)$
implies that $j \le k$. Moreover, 
due to \eqref{eq:ahlexpr} we have
$\ft{\gam}_{h,l} (u) = A_{h,l}(u)$, $u \in \R$,
which yields
\begin{equation}
\label{eq:gamhlfouriervan}
\ft{\gam}_{h,l} (u) = 0, \quad u \in U.
\end{equation}

It thus remains to show that for each $h \in \Lam-\Lam$ 
there is at least one  $l_0 = l_0(h)$ such that 
the distribution
$\gam_{h,l_0}$ is nonzero.
Indeed, given such $h$ we choose
$\lam_0 \in \Lam$ such that $\lam_0+h$ is also
in $\Lam$. Since $\varphi > 0$
we have 
$\supp(\al \cdot \varphi) = \Lam$,  hence
by \eqref{eq:Trexpr} 
there is a largest integer $j_0$
such that the coefficient
$b_{j_0}(\lam_0)$ is nonzero,
 and there is also at least one
integer $l_0$ such that
$b_{l_0}(\lam_0+h)$ is nonzero
(in particular we have $0 \le l_0 \le k$).
Let us show that for this choice of $l_0$
the distribution $\gam_{h,l_0}$ is nonzero. 
It would suffice to
verify that there is
precisely one triple $(p,q,j) \in J(k,l_0)$
with $j=j_0$ and such that
$b_q(\lam_0+h) \overline{b_p(\lam_0)}$ 
is nonzero, since this would imply that
in the sum
\eqref{eq:gamhldef} 
the coefficient 
of $\del_{\lam_0}^{(j_0)}$
is nonzero.  
Indeed, if $(p,q,j)$  is such 
a  triple  then  we have
$p \leq j_0$ due to  the maximality of $j_0$.
But at the same time
$p \geq j = j_0$,  so 
$p = j_0$. Since $p+q-j=l_0$, this implies in turn
that $q = l_0$.
It follows that $(j_0,l_0,j_0)$ is the unique triple in $J(k,l_0)$
with the properties above, and the coefficient 
$b_{l_0}(\lam_0+h) \overline{b_{j_0}(\lam_0)}$ is indeed nonzero.
We thus obtain that the distribution
$\gam_{h,l_0}$ is nonzero which concludes the proof.
\end{proof}


\section{Poisson type structure of the distribution}
\label{sec:poisson}

 In this section we complete the proof of
\thmref{thmA1} by establishing the following:

\begin{thm}
  \label{thmB2}
  Let $\alpha$ be a temperate distribution on $\R$.
Suppose that $\Lam = \supp(\alpha)$ is contained
in a finite union
of translates of a lattice $L$,
and that $S = \supp(\ft{\alpha})$
is a locally finite set.
Then $\alpha$ can be
represented in  the form
\eqref{eqA1.1}.
\end{thm}

\thmref{thmA1} follows as a consequence
of \thmref{thmB1} and \thmref{thmB2}.

The representation
\eqref{eqA1.1}
implies 
that the  distribution $\al$ 
can be obtained 
from the Poisson  summation formula by a finite 
number of basic operations:
 shifts, modulations, differentiations, 
multiplication by polynomials, and taking linear combinations.

\subsection{}
We  will use  the following two lemmas.

\begin{lem}
  \label{lem:confluent}
Let $z_1, \dots, z_N$ be distinct nonzero complex numbers,
let $k_1, \dots, k_N$ be positive integers,
and let $K = k_1 + \dots + k_N$.
For each integer
$m$, $0 \le m \le K-1$, 
and for each pair of integers
$(j,l)$ satisfying $1 \le j \le N$,
$0 \le l \le k_j - 1$, we denote
\begin{equation}
\label{eq:confdef}
M_{m,j,l} := (2 \pi i m)^l z_j^m.
\end{equation}
Let $M$ be a $K \times K$ matrix
whose
rows are indexed by $m$
and columns indexed by pairs
$(j,l)$, such that the 
entry at row $m$ and column
$(j,l)$ is $M_{m,j,l}$.
 Then  $M$ is invertible.
\end{lem}

This is due to the fact that the matrix
$M$ can be obtained by a finite number of
elementary column operations
from a \emph{confluent Vandermonde matrix}
(sometimes also referred to as a \emph{generalized} Vandermonde
matrix) which is known to be invertible,
see, for instance, \cite{Kal84}.
This type of matrix arises e.g.\ in the linear interpolation
problem asking to find a polynomial $p(z)$
of degree not greater than $K-1$ such that at each node
$z_j$ the values
$p(z_j), p'(z_j), \dots,
p^{(k_j-1)}(z_j)$ are prescribed.

\begin{lem}
  \label{lem:diagpolym}
Let $M(t)$ be a $(k+1) \times (k+1)$
matrix with entries
$M_{p,l}(t)$, $0 \le p,l \le k$,
defined by
\begin{equation}
\label{eq:diagup}
M_{p,l}(t) = \binom{p}{l} (2 \pi i t )^{p-l},
\quad  l \le p,
\end{equation}
and $M_{p,l}(t) = 0$, $p<l$.
Then the matrix $M(t)$ is invertible,
and the entries of the
inverse matrix $M^{-1}(t)$ 
are polynomials in $t$.
\end{lem}

This lemma is obvious since
$M(t)$ is a triangular matrix whose nonzero
entries are 
 polynomials in $t$, and
whose entries on the main diagonal are all equal to $1$.

\subsection{}
Next we prove 
 \thmref{thmB2}.
The approach is inspired by
 \cite[Section 7]{LO15}.

\begin{proof}[Proof of  \thmref{thmB2}]
By assumption we have
$\Lam \sbt L + \{ \tau_1, \dots, \tau_N\}$,
where $L \sbt \R$  is a lattice and the
$\tau_j$ are real numbers.
We may suppose that the 
$\tau_j$  are distinct modulo the
lattice $L$. Moreover,
by rescaling it would suffice to consider the
case  $L=\Z$.

By \lemref{lem:alphcoeffpolygrow} 
the distribution 
$\alpha$ can be represented in the form
\begin{equation}
\label{eq:alphdecom}
\alpha(x) = \sum_{p=0}^{k} 
\sum_{j=1}^{N} \mu_{j,p}^{(p)}(x - \tau_j)
\end{equation}
where each $\mu_{j,p}$ is a measure
supported on $\Z$,
\begin{equation}
\label{eq:aljpmes}
\mu_{j,p} = \sum_{\lam \in \Z}
a_{j,p}(\lam) \del_{\lam},
\end{equation}
and where the masses $a_{j,p}(\lam)$ satisfy
\begin{equation}
\label{eq:aljpmescoeff}
 \sum_{p=0}^{k}  \sum_{j=1}^{N}
|a_{j,p}(\lam)|  \leq C(1 + |\lam|)^n,
\quad \lam \in \Z,
\end{equation}
for certain constants $n$ and $C$.

It follows from 
\eqref{eq:aljpmescoeff}
that  each  measure
$\mu_{j,p}$ is a
 temperate distribution. We have
\begin{equation}
\label{eq:aldecomft}
\ft{\al}(t) = \sum_{p=0}^{k}  (2 \pi i t)^p  
\sum_{j=1}^{N} \ft{\mu}_{j,p}(t) e^{-2 \pi i t \tau_j}
\end{equation}
according to  \eqref{eq:alphdecom}.
Since $\mu_{j,p}$ is a  
measure supported on $\Z$,  
its Fourier transform $\ft{\mu}_{j,p}$ is a
$\Z$-periodic 
 distribution. This implies that 
for every $m \in \Z$
we have
\begin{equation}
\label{eq:alfttrnsint}
\ft{\al}(t+m) = 
\sum_{l=0}^{k} (2 \pi i m )^l  
\sum_{j=1}^{N}
\beta_{j,l}(t)
e^{-2 \pi i m \tau_j},
\end{equation}
where 
$\beta_{j,l}$  denotes
the temperate distribution
defined by
\begin{equation}
\label{eq:betdef}
\beta_{j,l}(t) := e^{-2 \pi i t \tau_j}
\sum_{p=l}^{k} \binom{p}{l} (2 \pi i t )^{p-l}
 \,  \ft{\mu}_{j,p}(t) 
\end{equation}
for $1 \le j \le N$ and $0 \le l \le k$.

If we now apply
\eqref{eq:alfttrnsint}
with
$m=0,1,2,\dots,(k+1)N-1$
then we
obtain
a system of $(k+1)N$ 
equations. We consider
this as a linear system  with unknowns
$\beta_{j,l}$. 
We then invoke \lemref{lem:confluent}
with $z_j := e^{-2 \pi i  \tau_j}$
(which are distinct nonzero complex numbers)
and with $k_j := k+1$ $(1 \le j \le N)$.
It follows from the
lemma that the system 
is invertible, and
hence there exist coefficients $b_{j,l,m}$ such that
\begin{equation}
\label{eq:betsol}
\beta_{j,l}(t) = 
\sum_{m=0}^{(k+1)N-1} b_{j,l,m} \, 
\ft{\al}(t+m).
\end{equation}
In other words, each 
$\beta_{j,l}$ is a finite linear combination
of integer translates of $\ft{\al}$.

Next we  apply
\eqref{eq:betdef}
with $l=0,1,\dots,k$ and consider
the obtained $k+1$ 
equations as 
 a linear system  
with unknowns 
$\ft{\mu}_{j,p}(t) \,
e^{-2 \pi i t \tau_j}$
$(0 \le p \le k)$.
By  \lemref{lem:diagpolym},
this system 
is invertible and the coefficients
of the inverse system 
are polynomials in $t$.
We conclude that 
there exist polynomials
$\chi_{p,l}(t)$ such that
\begin{equation}
\label{eq:mujpsol}
\ft{\mu}_{j,p}(t) =
e^{2 \pi i t \tau_j}
 \sum_{l=0}^{k} \chi_{p,l} (t) \,
\beta_{j,l}(t).
\end{equation}

Now  we use the assumption that
$S = \supp(\ft{\alpha})$
is a locally finite set. This
together with 
\eqref{eq:betsol} and
\eqref{eq:mujpsol}
 implies that  the distribution
$\ft{\mu}_{j,p}$ is supported on the
locally finite set
$S - \{0,1,2,\dots,(k+1)N-1\}$.
At the same time,  we know that
$\ft{\mu}_{j,p}$ is a
$\Z$-periodic distribution.
Hence 
$\ft{\mu}_{j,p}$
must have the form
\begin{equation}
\label{eq:aljftsimple}
\ft{\mu}_{j,p} = \nu_{j,p} \ast \sum_{\lam \in \Z} \del_\lam
\end{equation}
where $\nu_{j,p}$ is a 
distribution with finite support.
In turn this implies that
\begin{equation}
\label{eq:aljsimple}
\mu_{j,p}  =   \sum_{\lam \in \Z} \ft{\nu}_{j,p}(-\lam) \del_\lam,
\end{equation}
that is, the masses
$a_{j,p}(\lam)$
in \eqref{eq:aljpmes}
are given by
\begin{equation}
\label{eq:ajpvuft}
a_{j,p}(\lam)= \ft{\nu}_{j,p}(-\lam), \quad
\lam \in \Z.
\end{equation}

Finally we combine
\eqref{eq:alphdecom}  and
\eqref{eq:aljsimple} 
to conclude   that
\begin{equation}
\label{eq:alphfinal}
\alpha = 
\sum_{p=0}^{k} 
\sum_{j=1}^{N} 
   \sum_{\lam \in \Z} \ft{\nu}_{j,p}(-\lam) \del^{(p)}_{\lam+\tau_j}.
\end{equation}
We also observe that  each one of
the functions $\ft{\nu}_{j,p}(x)$ is a finite linear
combination of products of polynomials and
complex exponentials. Hence 
\eqref{eqA1.1} follows from
\eqref{eq:alphfinal}.
\end{proof}


\section{Remarks}

\subsection{}
We say that a set
$\Lam \sbt \R$ has \emph{bounded density} if it
satisfies the condition
\begin{equation}
\label{eq:bddsdef}
\sup_{x \in \R} \#(\Lam \cap [x, x +1)) < +\infty.
\end{equation}
This  holds if and only if
$\Lam$ is the union of a finite number of uniformly discrete sets. 

One can prove the following  version 
of \thmref{thmA1}, where 
the support $\Lam$ is not assumed
to be 
uniformly discrete but only to have
bounded density:

\begin{thm}
  \label{thmC1}
  Let $\alpha$ be a temperate distribution on $\R$
satisfying \eqref{eq:tudphiprod} and
\eqref{eq:tudsumcoefffin}, such that
the support
$\Lam = \supp(\alpha)$ has
bounded density, while the
spectrum 
$S = \supp(\ft{\alpha})$
is  uniformly discrete. Then the
conclusion of \thmref{thmA1} holds.
\end{thm}

This is an extension of 
  \cite[Theorem 2.2]{LO17} where
the result was proved in the case
where $\alpha$ 
and its Fourier transform $\ft{\al}$
are measures.
The proof is based on the fact that
\thmref{thmLO1} remains valid under the
weaker assumption that $\Lam$ is a 
relatively dense set of bounded density
(not assumed to be uniformly discrete),
see \cite[Lemma 6.3]{LO17}.

\subsection{}
There is an interesting question
as to whether the result in
\thmref{thmA1} holds 
 in several dimensions.
The problem is open even 
for measures:

\emph{Let $\Lam, S$ be two
uniformly discrete sets in $\R^d$, $d>1$.
Suppose that there is a  measure $\mu$,
$\supp(\mu) = \Lam$, whose
distributional Fourier transform
$\ft{\mu}$ is also a measure,
$\supp(\ft{\mu}) = S$. 
Does it follow that $\Lambda$ can be covered
by a finite union of translates of several
lattices?}

It was proved in \cite{LO15} that the answer
is affirmative if
$\mu$ is a \emph{positive} measure,
and in this case the support  $\Lam$ can in fact be covered
by a finite union of translates of a \emph{single}
lattice. However
an example in \cite[Section 2]{Fav16} shows that for signed
measures $\mu$ the support need not be
contained in a finite union of translates of a single lattice.


\end{document}